\documentclass[a4paper,12pt]{article}

\usepackage[left=2cm,right=2cm, top=2cm,bottom=2cm,bindingoffset=0cm]{geometry}

\usepackage[russian,english]{babel}
\usepackage{amsthm}
\usepackage{amsmath}
\usepackage{amssymb}
\usepackage{cite}

\usepackage{graphicx}

\newtheorem{thm}{Theorem}
\newtheorem{lemma}{Lemma}
\newtheorem{statement}{Proposition}
\theoremstyle{definition}
\newtheorem{ip}{Inverse problem}
\newtheorem{alg}{Algorithm}
\newtheorem{remark}{Remark}

\begin{document}

\begin{center} \bf \Large
On recovering quadratic pencils  with singular coefficients and 
entire functions in the boundary conditions
\end{center}

\begin{center}
Maria Kuznetsova \\
Saratov State University (SSU) \\
email: kuznetsovama@info.sgu.ru
\end{center}

{\bf Abstract.}  In the paper, we study an inverse spectral problem for quadratic pencils of the Sturm--Liouville operators with singular coefficients and entire functions in the boundary conditions. We prove that a subspectrum is sufficient for recovering the pencil if this subspectrum generates a complete functional system. As well, we obtain an algorithm solving the inverse problem and alternative conditions on the subspectrum. Finally, these results are applied to studying a partial inverse problem.

{\it Keywords:} inverse spectral problems, differential pencils, analytical dependence on the spectral parameter, singular coefficients, uniqueness theorem, partial inverse problems.

{ \it 2010 Mathematics Subject Classification:} 34A55, 34B05, 34B24.
%
%

\section{Introduction}
In this paper, we study an inverse spectral problem of recovering the coefficients in the quadratic differential pencil
\begin{equation} \label{pencil}
-y'' + q(x)y + 2\lambda p(x)y = \lambda^2 y, \quad x \in (0, \pi),
\end{equation}
where $p \in L_2(0, \pi)$ and $q \in W^{-1}_2(0, \pi).$ 
The latter means $q = \sigma'$ for some $\sigma \in L_2(0, \pi)$ in the sense of distributions. In particular, the class $W^{-1}_2$ contains the
Dirac delta-functions and the Coulomb-type singularities $\frac1x,$  which frequently appear in quantum mechanics~\cite{gesz}. Some results in inverse spectral problems for operators with singular coefficients were obtained in the works~\cite{sav,myk, manko,hryn1,pron2,bon-gaidel,higher,ignatiev, guliev, eckhardt}.

Further, we study equation~\eqref{pencil} by means of the so-called regularization approach~\cite{sav,myk, manko,hryn1,pron2,bon-gaidel,higher,ignatiev, guliev, eckhardt}.
Introducing the quasi-derivative $y^{[1]}= y' - \sigma y,$ we rewrite~\eqref{pencil} in the following equivalent form:
\begin{equation} \label{regularization}
\ell y + 2 \lambda p(x) y = \lambda^2 y,
\end{equation}
$$\ell y := -(y^{[1]})' - \sigma(x) y^{[1]} - \sigma^2(x) y, 
$$
 where $y^{[1]}, y \in AC[0, \pi],$ and $\ell y \in L_2(0, \pi).$
Actually, after this, equation~\eqref{pencil} with singular $q$ is reduced to a system of two equations having integrable coefficients. 
The expression $\ell y$ first was introduced in~\cite{savchuk} in order to study properties of Sturm--Liouville operators with singular potentials. 
We also mention the earlier work~\cite{pfaff}, where the term $y'-\sigma y$ was used for regularization. 

In contrast to the Sturm--Liouville equation, equation~\eqref{pencil} possesses non-linear dependence on the spectral parameter, which complicates studying inverse problems. By this reason, for quadratic pencils, the theory of inverse spectral problems is not so developed as for the Sturm--Liouville operator~\cite{march, levitan,poschel,novabook}. 
At the same time, inverse problems for equation~\eqref{pencil} arise in various applications, e.g., in modeling
interactions between colliding relativistic particles in quantum mechanics~\cite{jaulient} and in recovering mechanical systems vibrating in viscous media~\cite{yama}.

Inverse spectral problems of recovering the quadratic pencil~\eqref{pencil} were studied in~\cite{gas-gus, hryn1,pron2, bon-gaidel, but1, but2, nabiev, ahtyamov,gus-nab,gulsen} and other works. The singular case when the coefficients  $p \in L_2(0, \pi)$ and $q \in W^{-1}_2(0, \pi)$ was investigated in~\cite{hryn1,pron2,bon-gaidel}. Hryniv and Pronska~\cite{hryn1, pron2} obtained full solutions for inverse problems of recovering by one spectrum together with the weight numbers or by two spectra. However, they
considered only real-valued $p$ and $q$ under the conditions that guarantee the existence of a positive solution. 
The more general case of complex-valued coefficients was studied by Bondarenko and Gaidel~\cite{bon-gaidel}, who obtained, in particular, local solvability and stability of the inverse problem by the spectrum and the weight numbers.

Unlike the mentioned works, here we consider the boundary conditions with dependence on the spectral parameter
\begin{equation} \label{boundary conditions}
y(0) =0, \quad f_1(\lambda) y^{[1]}(\pi) + f_2(\lambda) y(\pi) = 0,\end{equation}
where $f_1(\lambda)$ and $f_2(\lambda)$ are known entire functions. Note that in the case of regular $q \in L_1(0, \pi),$ one can choose $\sigma \in AC[0, \pi]$ such that $y^{[1]}(\pi)=y'(\pi).$ We study the inverse problem of recovering the pencil~\eqref{regularization}, \eqref{boundary conditions} by its subspectrum $\{ \lambda_n \}_{n \ge 1}$ with a number containing information on~$\omega_0=\int_0^\pi p(t) \, dt:$ 
\begin{ip} \label{main ip}
Given $\{ \lambda_n \}_{n \ge 1}$ and $(\omega_0 \bmod 1),$ 
recover the coefficients $p$ and $\sigma.$
\end{ip}
Inverse problems for the quadratic pencils with the spectral parameter in boundary conditions were studied in~\cite{yang-yu, nabiev-azerb, freil, tang-trans,but-half,khalili}. However, the known works are limited to the polynomial dependence.  For the first time, we study an inverse spectral problem for the pencils~\eqref{regularization} with arbitrary entire functions in the boundary conditions.


The reason for using conditions~\eqref{boundary conditions} is that they allow to unify studying partial inverse problems. Such problems consist in recovering the operators coefficients on a part of the domain by some spectral data and the coefficients on the rest part. 
In recent years, partial inverse problems for the quadratic pencils have been actively studied, see~\cite{but-half, bon-graph, khalili, wang, wang-bon, bon-phys, yang, shieh, durak}. 
In~\cite{but-half, yang}, it was proved that quadratic pencils are recovered on the half of the interval by a spectrum on the whole interval, which is an analogue of the Hochstadt--Lieberman theorem for the Sturm--Liouville operator~\cite{hoh}. Another direction of studies are partial inverse problems on graphs~\cite{wang-bon, bon-graph, bon-phys, shieh}, wherein one has to reconstruct components of operators on a part of edges.

For the Sturm--Liouville equations, i.e. when $p=0,$ the boundary value problems with conditions~\eqref{boundary conditions} were considered in~\cite{OpenMath,MethAppl,sing-graph}.  In~\cite{OpenMath,MethAppl}, the inverse problem of recovering the potential $q \in L_2(0, \pi)$ by the subspectrum and the average of the potential was studied. Under the condition of completeness of a special functional sequence, the mentioned input data give spectral characteristics of two another boundary value problems without spectral parameter in the boundary conditions. 
This allows to obtain a uniqueness theorem and an algorithm for recovering, see~\cite{OpenMath}. Further study of the inverse problem concerning the issues of solvability and local stability was carried in the work~\cite{MethAppl}. 
In~\cite{sing-graph}, analogous results were obtained in  the case of singular $q \in W^{-1}_2(0, l),$  and the constructed theory was applied to studying a partial inverse problem on an arbitrary graph.
Let us note that the inverse problems considered in~\cite{OpenMath,MethAppl,sing-graph} are closely related to the reconstruction of the potential by the values of the Weyl function at a discrete set of points, see~\cite{horvat, torba}. 

In this paper, relying on the ideas of~\cite{OpenMath,MethAppl,sing-graph}, we develop an approach to inverse spectral problems for the quadratic pencil with singular $q.$
We find the conditions on a part of the spectrum under which Inverse Problem~\ref{main ip} has a unique solution. 
These conditions include the completeness of certain functional sequences. 
We prove the uniqueness theorem and obtain an algorithm for solving Inverse Problem~\ref{main ip}. We also establish conditions on the subspectrum  which yield the needed properties of the functional sequence.
Then, we provide an example of the partial inverse problem to which these results are applicable. 

The paper is organized as follows. 
In Section~\ref{prelim}, we introduce necessary objects and provide auxiliary results from other papers. 
In Section~\ref{results}, we obtain a uniqueness theorem and an algorithm for Inverse Problem~\ref{main ip}, see Theorem~\ref{main result} and Algorithm~\ref{algorithm}. 
In Section~\ref{sufficient}, we find alternative conditions on the subspectrum, see Theorem~\ref{completeness}. 
In Section~\ref{application}, the Hochstadt--Lieberman-type inverse problem reduces to Inverse problem~\ref{main ip}. As a consequence, we obtain a uniqueness theorem, which generalizes the result from~\cite{yang} to the singular case, see Theorem~\ref{theorem partial}.

\section{Preliminaries}
\label{prelim}
Consider the solutions $S(x, \lambda)$ and $C(x, \lambda)$ of equation~\eqref{regularization} satisfying the initial conditions
\begin{equation*}
S(0, \lambda) = 0, \; S^{[1]}(0, \lambda) = 1, \quad C(0, \lambda) = 1, \; C^{[1]}(0, \lambda) = 0.
\end{equation*}
These solutions exist and they are unique since~\eqref{regularization} can be rewritten as a first-order system with integrable coefficients, see~\cite{savchuk, pronska}.  By the same reason, $S(\pi, \lambda), $ $S^{[1]}(\pi, \lambda),$ $C(\pi, \lambda),$ and $C^{[1]}(\pi, \lambda)$ are entire functions. 
Further, we use the following representations being a consequence of~\cite[Theorem~1]{pronska}  (see also~\cite[Lemma~2.1]{bon-gaidel}):
\begin{equation} \left.
\begin{array}{c}
\displaystyle S(\pi, \lambda) = \frac{\sin \pi(\lambda - \omega_0)}{\lambda} + \frac{1}{\lambda} \int_{-\pi}^\pi \mathcal{K}(t) \exp(i \lambda t) \, dt, \\[3mm]
\displaystyle S^{[1]}(\pi, \lambda) = \cos \pi(\lambda - \omega_0) + \int_{-\pi}^\pi \mathcal{N}(t) \exp(i \lambda t) \, dt, 
\end{array} \right\}
\label{KN}
\end{equation}
where $\omega_0 = \frac{1}{\pi} \int_0^\pi p(s) \, ds$ and $\mathcal{K}, \mathcal{N} \in L_2(-\pi, \pi).$ As the piece of the input data for Inverse Problem~\ref{main ip}, we take the fractional part $(\omega_0 \bmod 1)$ of the number $\omega_0.$

Put $\mathbb Z_0 = \mathbb Z \setminus \{ 0 \}.$ The zeros of the function $S(\pi, \lambda)$ can be enumerated as $\{ \theta_k \}_{k \in \mathbb Z_0}$ counting their multiplicities in such way that they satisfy asymptotics
\begin{equation} \label{asym1}
\theta_k = k + \omega_0+\varkappa_k, \; k \in \mathbb Z_0, \quad \{\varkappa_k \}_{k \in \mathbb Z_0} \in \ell_2,
\end{equation}
see~\cite{pronska}.
Let us denote $\theta_0 = 0$ and
$$\mathbb{S}_\theta = \{n \in \mathbb Z \colon \forall k < n \quad \theta_k \ne \theta_n\}, \quad m_{\theta,n} = \#\{ k \in \mathbb Z \colon \theta_n = \theta_k\}.$$
Without loss of generality, we assume that  for $n>0,$ $\theta_n \ne 0,$ and equal values in the sequence $\{ \theta_k \}_{k \in \mathbb Z}$ follow each other:
\begin{equation} \label{assumption o}
\theta_{n} = \theta_{n+1} = \ldots = \theta_{n + m_{\theta,n}-1}, \quad n \in \mathbb{S}_\theta.
\end{equation}

We also introduce the notations
$$e(x, \lambda) = \exp(i \lambda x), \quad f^{<j>}(z) = \frac{1}{j!} \frac{d^{j}}{d \lambda^j} f(\lambda)\Big|_{\lambda=z}.$$
Hereinafter, we will use the following fact being a direct consequence of Lemma~2 in~\cite{b-stab}.
\begin{statement} \label{e basis}
Let  $\{ \theta_k \}_{k \in \mathbb Z_0}$ be arbitrary sequence satisfying asymptotics~\eqref{asym1}. 
Then, the functional sequence $\{ e^{<\nu>}(t, \theta_n)\}_{n \in \mathbb{S}_\theta, \nu = \overline{0, m_{\theta,n}-1}}$ is a Riesz basis in $L_2(-\pi, \pi).$ 
\end{statement}

Let us introduce the Weyl function $M(\lambda)= -\frac{C(\pi, \lambda)}{S(\pi, \lambda)}.$ In~\cite{bon-gaidel}, the following inverse problem was studied:
 \begin{ip} \label{sd ip}
 Given $\{ \theta_k \}_{k \in \mathbb Z_0}$ and the values
 $$M_{n+\nu} := \mathop{\mathrm Res}_{\theta=\theta_n} \, (\theta-\theta_n)^\nu M(\theta), \quad n \in \mathbb{S}_\theta, \; \nu = \left\{\begin{array}{cc}\overline{0, m_{\theta,n} - 1}, & \theta_n \ne 0,\\[2mm]
 \overline{0, m_{\theta,n} - 2}, & \theta_n = 0,
 \end{array}\right.$$
 find $p$ and $q.$ 
 \end{ip}
 We note that if $S(\pi, \lambda)$ and $S^{[1]}(\pi, \lambda)$ are known, then one can find the input data of Inverse problem~\ref{sd ip}.
 Indeed, it is easy to see that
\begin{equation} \label{wronski}
C(x, \lambda) S^{[1]}(x, \lambda) - C^{[1]}(x, \lambda) S(x, \lambda) = 1, \quad x \in [0, \pi].
\end{equation}
Using~\eqref{wronski} in $x=\pi,$ we obtain
$$ M(\lambda) S^{[1]}(\pi, \lambda) = \frac{1}{S(\pi, \lambda)} - C^{[1]}(\pi, \lambda).$$
Relation~\eqref{wronski} also yields $S^{[1]}(\pi, \theta_n) \ne 0,$ $n \in \mathbb Z_0.$ Then, we have
\begin{equation} \label{compute}
M_{n+\nu} = \frac{1}{S^{[1]}(\pi, \theta_n)} \, \mathop{\mathrm Res}_{\theta = \theta_n} \, \frac{(\theta - \theta_n)^\nu}{S(\pi, \theta)}, \quad n \in \mathbb{S}_\theta, \; \nu = \left\{\begin{array}{cc}\overline{0, m_{\theta,n} - 1}, & \theta_n \ne 0,\\[2mm]
 \overline{0, m_{\theta,n} - 2}, & \theta_n = 0.
 \end{array} \right.
\end{equation}
Since $\{ \theta_k\}_{k \in \mathbb Z_0}$ are the zeros of $S(\pi, \lambda),$ we uniquely construct the input data of  Inverse problem~\ref{sd ip}.
 
 In~\cite{bon-gaidel}, for the class $q \in W^{-1}_2(0, \pi)$ and $p \in L_2(0, \pi),$ a uniqueness theorem and an algorithm of solution were obtained for Inverse problem~\ref{sd ip}. 
For the regular case $q \in L_2(0, \pi)$ and $p \in W^1_2(0, \pi),$ the corresponding results were obtained earlier in~\cite{but1,but2}.
Further, we reduce Inverse problem~\ref{main ip} to Inverse Problem~\ref{sd ip} and apply the results from~\cite{bon-gaidel}. 
\section{A uniqueness theorem and an algorithm}
\label{results}
In this section, we consider the characteristic function and the subspectrum. We introduce the designations which allow working with multiple values in the subspectrum. After constructing a special functional sequence $\{ v_n \}_{n=0}^\infty,$ we obtain a uniqueness theorem and an algorithm.

First, we note that
a number $\lambda$ is an eigenvalue of the boundary value problem~\eqref{regularization},~\eqref{boundary conditions} if and only if it is a zero of the characteristic function
\begin{equation}
\Delta(\lambda) = f_1(\lambda) \eta_2(\lambda) + f_2(\lambda) \eta_1(\lambda), \label{Delta}
\end{equation}
where we denoted $\eta_1(\lambda) = S(\pi, \lambda)$ and $\eta_2(\lambda) = S^{[1]}(\pi, \lambda).$
Consider a sequence $\{ \lambda_k \}_{k \ge 1}$ such that $\Delta(\lambda_n) = 0$ and each $\lambda_k$ occurs in the sequence not more times than its multiplicity as zero of $\Delta(\lambda).$ We call $\{ \lambda_k \}_{k \ge 1}$ a subspectrum. 

Now, following the scheme applied in~\cite{OpenMath}, we construct certain functional sequence by the given subspectrum. 
Put $\lambda_0 := 0$ and introduce the notations
$$\mathbb{S}_\lambda = \{ n \ge 0 \colon \lambda_n \ne \lambda_k \quad \forall k \colon 0 \le k < n\}, \quad m_{\lambda,n} = \#\{ k \ge 0 \colon \lambda_k = \lambda_n\}.$$
We make the following assumption on $\{ \lambda_k \}_{k \ge 0}$ analogous to~\eqref{assumption o}:
$$\lambda_n = \lambda_{n+1} = \ldots = \lambda_{n+m_{\lambda, n}-1}, \quad n \in \mathbb{S}_\lambda.$$

Substituting \eqref{KN} into \eqref{Delta}, we obtain
\begin{multline} \label{lambdaDelta}
\lambda \Delta(\lambda) = \lambda f_1(\lambda)\Big( \lambda \cos \pi(\lambda - \omega_0) + \int_{-\pi}^\pi \mathcal{N}(t)  e(t, \lambda) \, dt\Big) + \\
+f_2(\lambda)\Big( \sin \pi(\lambda - \omega_0) + \int_{-\pi}^\pi \mathcal{K}(t) e(t, \lambda) \, dt\Big).
\end{multline}
Introduce the Hilbert space of complex-valued vector-functions
$$\mathcal{H} = L_2(-\pi, \pi) \oplus  L_2(-\pi, \pi).$$
For $g = [g_1, g_2]$ and $h=[h_1, h_2],$ the scalar product and the norm in $\mathcal H$ are given by the formulae
$$ (g, h)_{\mathcal H} = \int_{\pi}^\pi \big(\overline{g_1(t)} h_1(t) + \overline{g_2(t)} h_2(t)\big) \, dt, \quad \| h\|_{\mathcal H} = \sqrt{(h, h)}.$$
In particular, we have $u(t) := [\overline{\mathcal N}(t), \overline{\mathcal K}(t)] \in \mathcal H.$

Then, relation \eqref{lambdaDelta} yields
$$(u(\cdot), v(\cdot, \lambda))_{\mathcal H} = \lambda \Delta(\lambda) + w(\lambda),$$
where
$$v(t, \lambda) = [\lambda f_1(\lambda) e(t, \lambda), f_2(\lambda)e(t, \lambda)],$$ 
$$w(\lambda) = -f_1(\lambda) \lambda \cos \pi(\lambda-\omega_0) - f_2(\lambda) \sin \pi(\lambda-\omega_0).$$
Since $\{\lambda_k\}_{k\ge 0}$ are zeros of $\lambda \Delta(\lambda),$ we get
\begin{equation}
(u(\cdot), v^{<\nu>}(\cdot, \lambda_n))_{\mathcal H} = w^{<\nu>}(\lambda_n), \quad n \in \mathbb{S}_{\lambda}, \; \nu = \overline{0, m_{\lambda,n}-1}. \label{auxiliary basis}
\end{equation}
For $n=0$ and $\nu=0,$ by the analyticity of $S(\pi, \lambda)$ at $\lambda=0,$ we have a stronger relation
\begin{equation} \label{analyticity}
\int_{-\pi}^\pi \mathcal{K}(t) \, dt = \sin \pi \omega_0,
\end{equation}
see~\eqref{KN}.
For $n \in \mathbb{S}_\lambda$ and $\nu = \overline{0, m_{\lambda,n}-1},$ we denote 
$$v_{n+\nu}(t) = \left\{ \begin{array}{cc}  
v^{<\nu>}(t, \lambda_n), &   n+\nu>0,\\[2mm]
 [0, 1], & n=\nu=0,
\end{array}\right.
\quad 
w_{n+\nu}(t) = \left\{ \begin{array}{cc}  
w^{<\nu>}(\lambda_n), &   n+\nu>0,\\[2mm]
\sin \pi \omega_0, & n=\nu=0.
\end{array}\right.
 $$
 Obviously, when $n \in \mathbb{S}_\lambda$ and $\nu = \overline{0, m_{\lambda,n}-1},$ the index $k=n+\nu$ runs over all values in $\mathbb N \cup \{ 0\}.$
Using the introduced notations,  from~\eqref{auxiliary basis} and~\eqref{analyticity} we obtain
\begin{equation} \label{coefficientsRelation}
(u, v_{k})_{\mathcal H} = w_{k}, \quad k \ge 0.
\end{equation}
If the sequence $\{ v_n \}_{n=0}^\infty$ is complete in $\mathcal H,$ we can find $u = [\overline{\mathcal K}, \overline{\mathcal N}].$ Then, the functions $S^{[1]}(\pi, \lambda)$ and $S(\pi, \lambda)$ are known, and the input data of Inverse problem~\ref{sd ip} are constructed uniquely. 

Thus, we formulate the following assumption on the sequence $\{ v_k \}_{k=0}^\infty:$
\smallskip

\noindent{\it (C) The  sequence $\{ v_k \}_{k=0}^\infty$ is complete in $\mathcal H.$}
\smallskip

Now, we are ready to obtain a uniqueness theorem for Inverse Problem~\ref{main ip}.
Let $L(p, \sigma)$ be the boundary value problem~\eqref{regularization},~\eqref{boundary conditions} with arbitrary coefficients $\sigma, p \in L_2(0, \pi).$ 
Along with $L(p, \sigma),$ we consider another problem $L(\tilde p, \tilde\sigma)$ of the same form
but with other coefficients $\tilde p, \tilde q \in L_2(0, \pi)$ 
 Let us agree that, if a symbol $\alpha$ denotes an object related to $p$ and $q,$ then the symbol $\tilde \alpha$ with tilde will denote the analogous object related to $\tilde p$ and $\tilde  q.$
\begin{thm} \label{main result}
Let $\{\lambda_n\}_{n=1}^\infty$ and $\{\tilde\lambda_n\}_{n=1}^\infty$ be subspectra of the boundary value problems $L(p, \sigma)$ and $L(\tilde p, \tilde  \sigma),$ respectively. Suppose that the sequence $\{ v_n \}_{n=0}^\infty$ constructed by $\{ \lambda_n\}_{n=1}^\infty$ satisfies the condition (C). Then the equalities $\{ \lambda_n\}_{n=1}^\infty=\{ \tilde\lambda_n\}_{n=1}^\infty$ and $(\omega_0\bmod 1)=(\tilde\omega_0 \bmod 1)$ yield $p \equiv \tilde p,$ $\sigma \equiv \tilde \sigma.$
\end{thm}
\begin{proof}
I. Let us construct the sequences $\{ \tilde v_k\}_{k=0}^\infty$ and $\{ \tilde w_k\}_{k=0}^\infty$ by the subspectrum $\{ \tilde \lambda_n \}_{n=1}^\infty$ analogously to the sequences $\{v_k \}_{k=0}^\infty$ and $\{w_k\}_{k=0}^\infty.$ Similarly to~\eqref{coefficientsRelation}, we have
$$(\tilde u, \tilde v_k) = \tilde w_k, \; k \ge 0, \quad \tilde u := [\overline{\tilde{\mathcal N}}, \overline{\tilde{\mathcal K}}].$$ The condition $(\omega_0\bmod 1)=(\tilde\omega_0 \bmod 1)$ means that $\omega_0 = \tilde \omega_0 + m$ for some $m \in \mathbb Z.$ 
Since $\{ \lambda_n\}_{n=1}^\infty=\{ \tilde\lambda_n\}_{n=1}^\infty$ and $\omega_0=\tilde\omega_0+m,$ by the anti-periodicity of the functions $\sin \pi z$ and $\cos \pi z,$ the equalities $\tilde v_k = v_k$ and $\tilde w_k = (-1)^m w_k$ hold.
The latter means $(u - (-1)^m \tilde u, v_n) = 0,$ and, by condition (C), we arrive at $u = (-1)^m \tilde u$ in $\mathcal H.$ 
Then, we have $\mathcal K \equiv (-1)^m \tilde{\mathcal  K}$ and $\mathcal N \equiv (-1)^m \tilde{\mathcal N}.$ 

By formula~\eqref{KN}, the identities 
\begin{equation}
S(\pi, \lambda) =(-1)^m\tilde S(\pi, \lambda), \quad S^{[1]}(\pi, \lambda) = (-1)^m\tilde S^{[1]}(\pi, \lambda)
\label{S coincide}
\end{equation}
 are fulfilled. Then, the corresponding sequences $\{ \theta_k \}_{k \in \mathbb{Z}_0}$ and $\{ \tilde\theta_k \}_{k \in \mathbb{Z}_0},$ being the part of input data for Inverse Problem~2, coincide. 
Using~\eqref{compute} and the analogous relation for $\tilde M_k,$ 
 we get $M_k = \tilde M_k,$ $k \in \mathbb Z_0.$ By virtue of the uniqueness theorem~\cite[Theorem~2.9]{bon-gaidel} for Inverse Problem~\ref{sd ip}, we have proved that $p=\tilde p$ in $L_2(0, \pi)$ and $q = \tilde q$ in $W^{-1}_2(0, \pi).$ We also have $\omega_0 = \tilde \omega_0.$

II. Let us prove that $\sigma = \tilde \sigma$ as well. Since $q = \tilde q$ in $W^{-1}_2(0, \pi),$ for some $h \in \mathbb C,$ we have $\sigma = \tilde \sigma+h.$ Then, it is easy to see that 
$$S(x, \lambda) = \tilde S(x, \lambda), \ S^{[1]}(x, \lambda) = \tilde S^{[1]}(x, \lambda) - h S(x, \lambda), \quad x \in [0, \pi], \; \lambda \in \mathbb C.$$ 
Substituting $x = \pi,$ we obtain 
$\tilde S^{[1]}(\pi, \lambda) - S^{[1]}(\pi, \lambda) = h S(\pi, \lambda).$ Using this formula and~\eqref{S coincide} with $m=0,$ we have $h S(\pi, \lambda) \equiv 0.$ Taking arbitrary $\lambda \ne \theta_n,$ we obtain $h=0.$
 \end{proof}
 
For an algorithm, we need a stronger assumption than (C):
\smallskip

\noindent{\it (B) The  sequence $\{ v_k \}_{k=0}^\infty$ is an unconditional basis in $\mathcal H.$}
\smallskip

Under this condition, there exists  the basis $\{ v^*_k \}_{k=0}^\infty$ which is biorthogonal to $\{ v_k \}_{k=0}^\infty,$ and the series $\sum_{k=0}^\infty (f, v_k) \, v^*_k$ converges to $f $ for each $f \in \mathcal H,$ see~\cite[Ch.~VI]{gohberg}.
Based on this fact and the  above computations,  we obtain the following algorithm for solving Inverse Problem~\ref{main ip}.
 
 \begin{alg} \label{algorithm}
 Given the subspectrum $\{ \lambda_n \}_{n=1}^\infty$  and the value $(\omega_0 \bmod 1).$ 
 To recover $p$ and $\sigma,$ one should:
 \begin{enumerate}
 \item Construct $\{ v_k \}_{k=0}^\infty$ and $\{ w_k \}_{k=0}^\infty$ putting $\omega_0 := (\omega_0 \bmod 1).$ 
 \item Find the sequence $\{ v^*_k \}_{k=0}^\infty$ which is biorthogonal to $\{ v_k \}_{k=0}^\infty$ in $\mathcal H,$ and construct 
 $$u = [\overline{\mathcal N}, \overline{\mathcal K}] = \sum_{k=0}^\infty \overline{w_k} v^*_k.$$
 \item Construct the functions $S(\pi, \lambda)$ and $S^{[1]}(\pi, \lambda)$ by~\eqref{KN}.
 \item Find $\{ \theta_k \}_{k \in \mathbb Z_0}$ as the zeros of $S(\pi, \lambda)$ and $\{ M_k \}_{k \in \mathbb Z_0}$ by formula~\eqref{compute}.
 \item Find $p$ and $q$ as solution of Inverse problem~2, see Algorithm~2.8 in~\cite{bon-gaidel}.
 \item Put $\omega_0 = \frac{1}{\pi} \int_0^\pi p(s) \, ds.$  If $\omega_0 - (\omega_0 \bmod 1)$ is odd, multiply on $-1$ the functions $S(\pi, \lambda)$ and $S^{[1]}(\pi, \lambda)$ constructed in 3.
 \item Choose any $\tilde\sigma \in L_2(0, \pi)$ such that $\tilde\sigma'=q.$ 
 Construct $\tilde S^{[1]}(\pi, \lambda)$ corresponding to the boundary value problem $L(\tilde p, \tilde \sigma)$ with $\tilde p = p.$ Find $\sigma$ by the formula
 $$\sigma = \tilde \sigma + h, \quad h = \frac{\tilde S^{[1]}(\pi, \lambda) - S^{[1]}(\pi, \lambda)}{S(\pi, \lambda)}.$$
 \end{enumerate}
 \end{alg}
 
 \begin{remark} 
 In fact, for Algorithm~\ref{algorithm} we need only the basis property of $\{ v_k\}_{k=0}^\infty.$ However, practically it is easier to verify alternative conditions sufficient for (B), see Theorem~\ref{completeness}. 
 \end{remark}
 \section{Sufficient conditions for (C) and (B)}
 \label{sufficient}
In the previous section, we obtained the uniqueness theorem and the algorithm under the conditions formulated for the sequence $\{ v_k \}_{k=0}^\infty.$ 
For construction, this sequence requires the subspectrum $\{ \lambda_n \}_{n = 1}^\infty$ and the pair of the functions $f_1(\lambda),$ $f_2(\lambda).$ It is easier to consider another sequence constructed only by the subspectrum and verify the following independent conditions:
 
 \smallskip
 \noindent{\it (S) For $n \in \mathbb N,$ the functions $f_1(\lambda)$ and $f_2(\lambda)$ do not vanish simultaneously in $\lambda_n.$}
  \smallskip
 
 \noindent{\it (C2) The functional sequence $\{ e^{<\nu>}(t, \lambda_n)\}_{n\in \mathbb{S}_\lambda, v=\overline{0, m_{\lambda,n}-1}}$ is complete in $L_2(-2\pi, 2\pi).$}
  \smallskip
 
 \noindent{\it (A) For the subspectrum $\{ \lambda_n \}_{n \ge 1},$ the following formulae hold:
 ${\mathrm Im}\,\lambda_n = O(1),$ 
and $m_{\lambda,n} = 1$ for $n \ge n_0,$ $n \in \mathbb N.$}
  \smallskip
 
 \noindent{\it (B2) The functional sequence $\{ e^{<\nu>}(t, \lambda_n)\}_{n\in \mathbb{S}_\lambda, v=\overline{0, m_{\lambda,n}-1}}$ is a Riesz basis in $L_2(-2\pi, 2\pi).$} 
  \smallskip

The main result of the section is the following theorem.
  \begin{thm} \label{completeness}
(i) Under conditions (S) and (C2), condition (C) holds.

\noindent{(ii) Under conditions (S), (A), and (B2), condition (B) holds.}
\end{thm}
By the point (i) of this theorem, conditions~(S) and~(C2) are sufficient for the uniqueness of recovery. By its point~(ii), conditions~(S),~(A), and~(B2) are sufficient for Algorithm~\ref{algorithm}.

To prove Theorem~\ref{completeness}, we need several auxiliary results.

 \begin{lemma} \label{completeEquiv}  Let us introduce the functions
 $$g(t, \lambda)= [\lambda \eta_1(\lambda) e(t, \lambda), -\eta_2(\lambda) e(t, \lambda)],$$
 $$g_{n+\nu}(t)= \left\{\begin{array}{cc} g^{<\nu>}(t, \lambda_n), & n+\nu > 1, \\[2mm]
[0, 1], & n=\nu=0,
 \end{array} \right.
 \quad  n \in \mathbb S_\lambda, \; \nu=\overline{0, m_{\lambda,n}-1}.$$
Under condition~(S), the sequence $\{ v_n \}_{n=0}^\infty$ is complete in $\mathcal H$ whenever $\{ g_n \}_{n=0}^\infty$ is complete in $\mathcal H.$
 \end{lemma}
Lemma~\ref{completeEquiv} is proved analogously to Lemma~3.2  in~\cite{OpenMath}.
 
\begin{statement} \label{Paley-Wiener}
Let $G(\lambda)$ be an entire function such that
$$|G(\lambda)| \le C \exp(|{\mathrm Im}\, \lambda|2\pi), \quad G \in L_2(\mathbb R),$$
$$G^{<\nu>}(\lambda_n) = 0, \quad n \in \mathbb{S}_\lambda, \; \nu = \overline{0, m_{\lambda,n}-1}.$$
If condition~(C2) holds, then $G(\lambda) \equiv 0.$
\end{statement}
\begin{proof}
By the Paley-Wiener theorem,
$$G(\lambda) = \int_{-2\pi}^{2\pi} \Phi(u) \exp(i \lambda u) \, du =  \int_{-2\pi}^{2\pi} \Phi(u)  e(u, \lambda) \, du,$$
where $\Phi \in L_2(-2\pi, 2\pi).$ This yields 
$$G^{<\nu>}(\lambda_n) = \int_{-2\pi}^{2\pi} \Phi(u) e^{<\nu>}(u, \lambda_n) \,du = 0, \quad n \in \mathbb{S}_\lambda, \; \nu = \overline{0, m_{\lambda,n}-1}.$$
Since the system $\{ e^{<\nu>}(t, \lambda_n)\}_{n\in \mathbb{S}_\lambda, v=\overline{0, m_{\lambda,n}-1}}$ is complete, we have $\Phi \equiv 0,$ and $G \equiv 0.$
\end{proof}

\begin{lemma} \label{g^0}
Denote 
$$g^0(x, \lambda) =  \big[ \exp(i \lambda x)\sin \pi(\lambda - \omega_0),  -\exp(i \lambda x)\cos \pi(\lambda - \omega_0) \big].$$
Let a sequence of distinct numbers $\{ \tau_n \}_{n \ge 0}$ be such that $\{ \exp(i \tau_n x)\}_{n \ge 0}$ is a Riesz basis in $L_2(-2 \pi, 2\pi).$ 
Then, $g^0_n(x) := g^0(x, \tau_n),$ $n\ge0,$ constitute a Riesz basis in $\mathcal H.$
\end{lemma}
\begin{proof}
I. First, suppose $\omega_0=0,$ then 
$$g^0_n(t) = [\exp(i \tau_n t)\sin \pi \tau_n, -\exp(i \tau_n t)\cos \pi \tau_n], \quad n \ge 0.$$
Let us prove completeness of the system $\{ g_n^0\}_{n \ge  0}.$

Consider $h \in \mathcal H$ such that $(h, g_n^0)=0$ for all $n \ge 0.$ Introduce the function
\begin{equation} \label{G_0}
G(\lambda) = -\int_{-\pi}^\pi \big(\overline{h_1(t)}\sin \pi \lambda - \overline{h_2(t)}\cos \pi \lambda\big) \exp(i \lambda t) \,dt,
\end{equation}
which turns zero in $\lambda = \tau_n.$ The function $G(\lambda)$ and  the numbers $\lambda_n = \tau_n,$ $n \ge 0,$ satisfy the conditions of Proposition~\ref{Paley-Wiener}, and we get $G(\lambda) \equiv 0.$ Substituting $\lambda=n$ into~\eqref{G_0}, we obtain $h_2 \equiv 0,$ and, as a consequence, $h_1 \equiv 0.$ Thus, we have proved $h\equiv0,$ and the sequence $\{ g_n^0\}_{n \ge 0}$ is complete in $\mathcal H.$

Now, to prove $\{ g_n^0\}_{n \ge 0}$ is a Riesz basis, it is sufficient to establish the two-side inequality
\begin{equation} \label{Riesz}
M_1 \sum_{n=0}^{N_0} |b_n|^2 \le \left\| \sum_{n=0}^{N_0} b_n g_n^0  \right\|_{\mathcal H} \le M_2  \sum_{n=0}^{N_0} |b_n|^2
\end{equation}
 with some fixed constants $M_1$ and $M_2$ for any $N_0 \in \mathbb N,$ $b_n \in \mathbb C.$
For this purpose, we compute
$$(g_n^0, g_k^0)_{\mathcal H} = (\overline{\sin \pi \tau_n}\sin \pi \tau_k + \overline{\cos \pi \tau_n}\cos \pi \tau_k) \int_{-\pi}^\pi \overline{\exp (i \tau_n t)} \exp (i \tau_k t) \, dt.$$
Replacing $\sin z$ and $\cos z$ by combinations of exponential functions, after manipulations, we obtain
$$(g_n^0, g_k^0)_{\mathcal H} = \frac12 \big(\exp(i \tau_n t), \exp(i \tau_k t)\big)_{L_2(-2\pi, 2 \pi)}.$$ 
Hence, it follows that
$$\left\| \sum_{n=0}^{N_0} b_n g_n^0 \right\|_{\mathcal H} =  \frac{1}{\sqrt2}\left\| \sum_{n=0}^{N_0} b_n \exp(i \tau_n t) \right\|_{L_2(-2\pi, 2 \pi)}.$$
Since $\{ \exp(i \tau_n x)\}_{n \ge 0}$ is a Riesz basis, the latter yields~\eqref{Riesz}.

II. Consider the case of arbitrary $\omega_0.$ Introduce the sequences 
$$\alpha_n = \tau_n - \omega_0, \ \tilde g_n(x) = \big[ \exp(i \alpha_n x) \sin \pi \alpha_n, -  \exp(i \alpha_n x)\cos \pi \alpha_n \big], \quad n \ge 0.$$ 
By part I, if $\{ \exp(i \alpha_n t)\}_{n\ge0}$ is a Riesz basis, so is $\{ \tilde g_n\}_{n\ge0}.$
Note that multiplication by the function $\exp(-i \omega_0 t)$ does not change the Riesz basis property. By this reason, $\{ \exp(i \alpha_n t)\}_{n\ge0}$ is a Riesz basis, and, in turn, $g_n^0(t) = \exp(i \omega_0 t) \tilde g_n(t),$ $n \ge 0,$ constitute a Riesz basis.
\end{proof}

\begin{proof}[Proof of Theorem~\ref{completeness}] (i) Assume that (S) and (C2) hold.  Let us prove the completeness of $\{ g_n \}_{n=0}^\infty$ in $\mathcal H.$ 
Consider $h = [h_1, h_2] \in \mathcal H$ such that $(h, g_k)_{\mathcal H} = 0,$  $k \ge 0.$ Construct the function
$$G(\lambda) = \int^\pi_{-\pi} (h_1(t) \lambda \eta_1(\lambda)- h_2(t) \eta_2(\lambda)) e(t, \lambda)\, dt.$$ 
Then, $G(\lambda)$ satisfies the conditions of Proposition~\ref{Paley-Wiener}, and $G(\lambda) \equiv 0.$ 

Substituting $\lambda=\theta_n$ into $G(\lambda),$ after differentiation,  we obtain 
$$\sum_{j=0}^\nu \int_{-\pi}^\pi h_2(t) \eta^{<j>}_2(\theta_n) e^{<\nu-j>}_2(t, \theta_n) \, dt = 0, \quad n \in \mathbb S_\theta, \, \nu=\overline{0, m_{\theta,n}-1}.$$
By induction, since $\eta_2(\theta_n) \ne 0,$ we get
$$\int_{-\pi}^\pi e^{<\nu>}(t, \theta_n) h_2(t) \, dt = 0,  \quad  n \in \mathbb S_\theta, \, \nu=\overline{0, m_{\theta,n}-1}. $$
Using Proposition~\ref{e basis}, we arrive at $h_2 \equiv 0$ and, hence, $h_1 \equiv 0.$ These identities mean the completeness of $\{ g_n\}_{n=0}^\infty$ in $\mathcal H.$ 
By Lemma~\ref{completeEquiv}, we have proved (C).

(ii) Now, we prove that (S), (A), and (B2) yield (B). From formula~\eqref{KN} it follows that
$$g(t, \lambda) = g^0(t, \lambda) + \exp(i \lambda t)\int_{-\pi}^\pi \exp(i \lambda\xi) f(\xi) \, d\xi, \quad f = [f_1, f_2] \in \mathcal H.$$
Referring to conditions~(A), we have $|\exp (i \lambda_n t)|=O(1),$ $n \ge n_0.$ 
Since $\{ \exp (i \lambda_n \xi)\}_{n \ge n_0}$ is a part of the Riesz basis in (B2), for $j=0,1,$ we have $\{ \int_{-\pi}^\pi \exp(i \lambda_n\xi) f_j(\xi) \, d\xi \}_{n \ge n_0} \in \ell_2.$ Combining these formulae, we obtain
\begin{equation} \label{quad+}
\left\{ \left\|  \exp(i \lambda_n t)\int_{-\pi}^\pi \exp(i \lambda_n\xi) f(\xi) \, d\xi \right\|_{\mathcal H}\right\}_{n \ge n_0} \in \ell_2.
\end{equation}

Put $\tau_n=\lambda_n$ for $n\ge n_0,$ while for the other $n \in \mathbb N,$ the numbers $\tau_n$ are chosen such that $\tau_n \ne \tau_k$ for all $n\ne k.$   
Then, the system $\{ \exp(i \tau_n t)\}_{n\ge 0}$ differs only by a finite number of elements from the Riesz basis in (B2). It is easy to see that this system is complete, see corollary from Theorem~2 in~\cite[Appendix III]{levin}.
By~\cite[Proposition~1.8.5]{novabook},  the system $\{ \exp(i \tau_n t)\}_{n\ge 0}$ is a Riesz basis as a complete sequence being quadratically close to a Riesz basis. 
Then, the system $\{ g_n^0 \}_{n\ge 0}$ constructed in Lemma~\ref{g^0} is also a Riesz basis.

In part (i) of the theorem we proved the completeness of the system $\{ g_n\}_{n \ge 0}.$ 
From formula~\eqref{quad+} it follows that $\{ g_n\}_{n \ge 0}$ is quadratically close to $\{ g_n^0 \}_{n\ge 0}.$
Then, the system $\{ g_n\}_{n \ge 0}$ is a Riesz basis in $\mathcal H.$

 In~\cite[Lemma~3.1]{OpenMath}, under condition~(S), the following relations were obtained:
 \begin{equation*}
\eta_j^{<\nu>}(\lambda_n) = (-1)^{j-1} \sum_{k=0}^\nu C_{n,k} f_j^{<\nu>}(\lambda_n), \quad n \in {\mathbb S}_\lambda, \ \nu = \left\{\begin{array}{cc}\overline{0, m_{\lambda,n} - 1}, & n > 0,\\[2mm]
 \overline{0, m_{\lambda,n} - 2}, & n = 0,
 \end{array} \right.
\end{equation*}
where $j=1,2$ and $C_{n,0} \ne 0.$ By condition (A), for $n \ge n_0,$  we have 
 $$\eta_j(\lambda_n) = (-1)^{j-1} C_{n,0} f_j(\lambda_n), \ j=1,2,\quad  g_n = C_{n,0} v_n.$$
 Put $C_{n,0} = 1$ for $n \notin \mathbb S_\lambda$ and consider the functions $\tilde v_n=C_{n,0}^{-1} v_n,$ $n \ge 0.$ Obviously, the system $\{ \tilde v_n\}_{n \ge 0},$ being complete and quadratically close to $\{ g_n\}_{n \ge 0},$ is a Riesz basis. 
It is known that a Riesz basis is an unconditional, see~\cite[Theorem 2.2]{gohberg}. Then, $v_n = C_{n,0} \tilde v_n,$ $n \ge 0,$ constitute an unconditional basis.
\end{proof}
\section{Application to a partial inverse problem}
\label{application}
Let us apply the obtained results to studying the inverse problem of recovering the pencil on the half of an interval by one spectrum.
We study the boundary value problem
\begin{equation} \label{doubled interval}
-y'' + q(x) y + 2 \lambda p(x) y = \lambda^2 y, \quad x \in (0, 2 \pi),
\end{equation} 
\begin{equation} \label{dir2}
y(0) = y(2 \pi) = 0,
\end{equation}
where $p \in L_2(0, 2 \pi)$ and $q \in W^{-1}_2(0, 2\pi).$ Considering arbitrary $\sigma \in L_2(0, 2\pi)$ such that $q=\sigma',$ we rewrite~\eqref{doubled interval} in the form~\eqref{regularization}.
From~\eqref{asym1}, we have that  the eigenvalues $\{ \mu_k \}_{k \in \mathbb Z_0}$ of the boundary value problem~\eqref{doubled interval},~\eqref{dir2} satisfy the asymptotics
\begin{equation} \label{asym2}
\mu_k = \frac{k}{2} + \frac{1}{2\pi}\int_0^{2\pi} p(t) \,dt + \varkappa_k, \; k \in \mathbb Z_0, \quad \{\varkappa_k\}_{k \in \mathbb Z_0} \in \ell_2.
\end{equation}
Note that since the quasi-derivative is absent in boundary conditions~\eqref{dir2}, the eigenvalues $\{ \mu_k \}_{k \in \mathbb Z_0}$ do not depend on the choice of $\sigma \in L_2(0, 2\pi)$ up to a constant function.

We assume that the coefficients $p$ and $q$ are known on $(\pi, 2\pi).$
In~\cite{yang}, in the regular case of $p \in W^1_2(0, 2\pi)$ and $q \in L_2(0, 2\pi),$ the following inverse problem was studied:
\begin{ip} \label{partial ip}
Given $\{ \mu_k \}_{k \in \mathbb Z_0}$ along with $p$ and $q$ on $(\pi, 2\pi),$ 
recover the coefficients $p$ and $q$ on the interval $(0, \pi).$
\end{ip}

For the first time, Inverse Problem~\ref{partial ip} in the regular case was studied by Buterin~\cite{but-half}. He proved a uniqueness theorem  and obtained an algorithm for solving the inverse problem  for the boundary conditions possessing the spectral parameter. For the Dirichlet boundary conditions, the uniqueness theorem was proved by Yang and Zettl~\cite{yang}, but under an additional assumption of the spectrum simplicity. In the present paper, by reducing Inverse Problem~\ref{partial ip} to Inverse Problem~\ref{main ip}, we prove the uniqueness theorem in the singular case of coefficients $p\in L_2(0, 2 \pi)$ and $q  \in W^{-1}_2(0, 2\pi)$ without assuming the simplicity of the spectrum.

For $x \in (\pi, 2\pi),$ we take arbitrary $\sigma(x)$ such that $q=\sigma'$ and introduce the solution $\varphi(x, \lambda)$ of~\eqref{doubled interval} satisfying the initial conditions
$$\varphi(2 \pi, \lambda) =  0, \quad \varphi^{[1]}(2 \pi, \lambda) = 1.$$
A number $\lambda$ is an eigenvalue of the problem~\eqref{doubled interval},~\eqref{dir2}  if and only if  the functions $\varphi(x, \lambda)$ and $S(x, \lambda)$ are linearly dependent.  The latter holds whenever
$$\varphi^{[1]}(\pi, \lambda) S(\pi, \lambda) -  S^{[1]}(\pi, \lambda) \varphi(\pi, \lambda) = 0.$$
It is clear that $\varphi(\pi, \lambda)$ and $\varphi^{[1]}(\pi, \lambda)$ are known entire functions.
Then, the eigenvalues $\{ \mu_k \}_{k \in \mathbb Z_0}$ of~\eqref{doubled interval},~\eqref{dir2} coincide with the eigenvalues  of~\eqref{regularization},~\eqref{boundary conditions}, where we put $f_1(\lambda)=-\varphi(\pi, \lambda)$ and $f_2(\lambda) = \varphi^{[1]}(\pi, \lambda).$
This fact and the results from the previous section allow us to obtain a uniqueness theorem. We use the same agreement regarding the symbols $\alpha$ and $\tilde \alpha$ as it was done before Theorem~\ref{main result}.
\begin{thm} \label{theorem partial}
Let $q = \tilde q$ in $W^{-1}_2(\pi, 2\pi),$ $p=\tilde p$ in $L_2(\pi, 2\pi),$ and $\{ \mu_k \}_{k \in \mathbb Z_0} = \{ \tilde\mu_k \}_{k \in \mathbb Z_0}.$ Then, the identities $q \equiv \tilde q$ and $p \equiv \tilde p$ hold on $(0, 2 \pi).$
\end{thm}
\begin{proof} 
Putting $f_1(\lambda)=-\varphi(\pi, \lambda),$  $f_2(\lambda) = \varphi^{[1]}(\pi, \lambda),$ we consider two boundary value problems $L(p, \sigma)$ and $L(\tilde p, \tilde \sigma).$  Then, the sequences $\{ \lambda_k \}_{k \ge 1}=\{ \mu_k \}_{k \in \mathbb Z_0}$ and $\{ \tilde\lambda_k \}_{k \ge 1}=\{ \tilde\mu_k \}_{k \in \mathbb Z_0}$ are their spectra, respectively.

Taking into account formula~\eqref{asym2}, one can uniquely reconstruct $(\omega_0\bmod 1)$ by the spectrum $\{ \mu_k \}_{k \in \mathbb Z_0}$ along with the known mean value of $p(x)$ on $(\pi, 2\pi).$  
By the conditions of the theorem, the latter yields $(\omega_0\bmod 1)=(\tilde\omega_0\bmod 1).$
Let us prove that condition (C) holds. Then, the statement of the theorem will follow from Theorem~\ref{main ip}. 

The functions $\varphi(\pi, \lambda)$ and $\varphi^{[1]}(\pi, \lambda)$ can not turn into zero for the same $\lambda,$ otherwise $\varphi(\pi, \lambda)$ would be a trivial solution.
Thus, condition (S) is fulfilled. By extending the segment in Proposition~\ref{e basis} to $(-2\pi, 2 \pi)$ and substituting $ \{\theta_n \}_{n \in \mathbb Z_0} = \{ \mu_n \}_{n \in \mathbb Z_0},$  we assert condition $(B2),$ which is stronger than (C2). By virtue of Theorem~\ref{completeness}, condition (C) holds.
\end{proof}
In the proof, for the corresponding boundary value problem~\eqref{regularization},~\eqref{boundary conditions} and its spectrum $\{ \lambda_n \}_{n\ge1},$  we obtained (S) and (B2). Condition (A) obviously follows from~\eqref{asym2}. Then, by Theorem~\ref{completeness}, condition~(B) holds, and we can apply Algorithm~\ref{algorithm} for recovery of $p$ and $q$ on~$(0, \pi).$


\begin{remark} Analogously one can study other types of boundary conditions and other partial inverse problems, including the ones from~\cite{bon-graph,but-half,wang-bon,wang,bon-phys,yang}.
\end{remark}

{\bf Aknowledgement.} This work was financially supported by project no.~21-71-10001 of the Russian Science Foundation, https://rscf.ru/en/project/21-71-10001/.


\begin{thebibliography}{9}
\bibitem{gesz} Albeverio S., Gesztesy F., Hoegh-Krohn R., Holden H. Solvable Models in Quantum Mechanics. 2nd revised ed. AMS Chelsea Publishing, Providence, RI, 2005.

\bibitem{myk} Hryniv R.O., Mykytyuk Y.V. Inverse spectral problems for Sturm-Liouville operators with singular potentials, Inverse Problems 19 (2003), no. 3, 665--684.
https://doi.org/10.1088/0266-5611/19/3/312

\bibitem{sav} Savchuk A.M., Shkalikov A.A. Inverse problem for Sturm-Liouville operators with distribution potentials: Reconstruction from two spectra, Russ. J. Math. Phys. 12 (2005), no. 4, 507--514.

\bibitem{higher} Bondarenko N.P. Inverse Spectral Problems for Arbitrary-Order Differential Operators with Distribution Coefficients. Mathematics 9 (2021), no. 22, 2989. https://doi.org/10.3390/math9222989 

\bibitem{hryn1} Hryniv R., Pronska N. Inverse spectral problems for energy-dependent Sturm-Liouville
equations, Inverse Problems 28 (2012), 085008 (21 pp).
https://doi.org/10.1088/0266-5611/28/8/085008

\bibitem{pron2} Pronska N. Reconstruction of energy-dependent Sturm-Liouville equations from two spectra, Int. Eqns. Oper. Theory 76 (2013), no. 3, 403--419.
https://doi.org/10.15330/cmp.5.2.315-325 

\bibitem{manko} Hryniv R.O., Manko S.S. Inverse scattering on the half-line for energy-dependent Schrödinger equations, Inverse Problems 36 (2020), no. 9,  article 095002.
https://doi.org/10.1088/1361-6420/aba416

\bibitem{bon-gaidel} Bondarenko N.P., Gaidel A.V. Solvability and stability of the inverse problem for the quadratic
differential pencil, Mathematics 9 (2021), no. 20, article 2617. https://doi.org/10.3390/math9202617  (27 pp.)

\bibitem{ignatiev} Freiling G., Ignatiev M.Y., Yurko V. A. An inverse spectral problem for Sturm--Liouville operators with singular potentials on star-type graph, Proc. Symp. Pure Math. 77 (2008), 397--408.
https://doi.org/10.1090/pspum/077 

\bibitem{eckhardt} Eckhardt J., Gesztesy F., Nichols R., Teschl G. Supersymmetry and Schrödinger-type operators with distributional matrix-valued
potentials. J. Spectr. Theory 4 (2014), issue 4, 715--768. 
https://doi.org/10.4171/JST/84

\bibitem{guliev} Guliyev N.J. Schrödinger operators with distributional potentials and boundary conditions dependent on the eigenvalue
parameter. J. Math. Phys. 60 (2019), issue 6, 063501. 
https://doi.org/10.1063/1.5048692

\bibitem{savchuk} Savchuk  A.M., Shkalikov A.A. Sturm-liouville operators with singular potentials. Math Notes 66 (1999), no. 6, 741--753. https://doi.org/10.1007/BF02674332

\bibitem{pfaff} Pfaff  R. Gewohnliche lineare differentialgleichungen zweiter ordnung mit Distributionskoeffizient. Arch. Math. 32 (1979), no. 5, 469--478.

\bibitem{march} Marchenko V.A. Sturm-Liouville Operators and Their Applications, Naukova Dumka, Kiev, 1977 (Russian); English transl., Birkhauser, 1986.

\bibitem{levitan} Levitan B.M. Inverse Sturm-Liouville Problems, Nauka, Moscow, 1984 (Russian); English transl., VNU Sci. Press, Utrecht, 1987.

\bibitem{poschel} Pöschel J.; Trubowitz E. Inverse Spectral Theory, New York, Academic Press, 1987.

\bibitem{novabook} Freiling G., Yurko V.A. Inverse Sturm--Liouville Problems and Their Applications, NOVA Science Publishers, New York, 2001.

\bibitem{jaulient} Jaulent M., Jean C. The inverse s-wave scattering problem for a class of potentials depending on energy. Comm. Math. Phys. 28 (1972), 177--220.
https://doi.org/10.1007/BF01645775

\bibitem{yama} Yamamoto M. Inverse eigenvalue problem for a vibration of a string with viscous drag. J. Math. Anal. Appl. 152 (1990), iss.~1, 20--34.
https://doi.org/10.1016/0022-247X(90)90090-3

\bibitem{gas-gus} Gasymov M.G., Guseinov G.Sh. Determination of diffusion operator from spectral data,
Akad. Nauk Azerb. SSR. Dokl. 37 (1981), 19--23 (Russian).

\bibitem{but1} Buterin S.A., Yurko V.A. Inverse spectral problem for pencils of differential operators on
a finite interval, Vestnik Bashkir. Univ. 11 (2006), no. 4, 8--12 (Russian).

\bibitem{nabiev} Nabiev  I.M., Shukurov A.Sh. Solution of inverse problem for the diffusion operator in a symmetric case, Izv. Saratov. Univ. Ser. Mat. Mekh. Inf. 9(1)(2009), no. 4 , 36--40 (in Russian).  https://doi.org/10.18500/1816-9791-2009-9-4-1-36-40

\bibitem{ahtyamov} Akhtyamov A.M., Sadovnichy V.A., Sultanaev Ya.T. Inverse problem for the diffusion operator with symmetric functions and general boundary conditions, Eurasian Math.J. 8 (2017), no. 1, 10--22.

\bibitem{gus-nab} Guseinov I. M., Nabiev I. M. The inverse spectral problem for pencils of differential
operators, Sb. Math. 198 (2007), no. 11, 1579--1598.
http://dx.doi.org/10.1070/SM2007v198n11ABEH003897

\bibitem{but2} Buterin S.A., Yurko V.A. Inverse problems for second-order differential pencils with
Dirichlet boundary conditions, J. Inverse Ill-Posed Probl. 20 (2012), no. 5--6, 855--881.
https://doi.org/10.1515/jip-2012-0062

\bibitem{gulsen} Gulsen T., Panakhov E.S. On the isospectrality of the scalar energy-dependent
Schrödinger problems, Turkish J. Math. 42 (2018), no. 1, 139--154.


\bibitem{yang-yu} Yang Ch.F., Yu X.J. Determination of differential pencils with spectral parameter dependent boundary conditions from interior spectral data, Math. Meth. Appl. Sci. 37 (2013), no. 6, 860--869.  https://doi.org/10.1002/mma.2844

\bibitem{nabiev-azerb} Ibadzadeh Ch.G., Mammadova L.I., Nabiev I.M. Inverse Problem of Spectral Analysis for Diffusion
Operator with Nonseparated Boundary Conditions
and Spectral Parameter in Boundary Condition, Azerbaijan Journal of Mathematics 9 (2019), no. 1, 171--189.

\bibitem{freil} Freiling G., Yurko V. Determination of Singular Differential Pencils from the Weyl Function, Advances in Dynamical Systems and Applications 7 (2012), no. 2, 171--193. https://campus.mst.edu/adsa/contents/v7n2p3.pdf

\bibitem{tang-trans} Yang Ch.F. Uniqueness theorems for differential pencils with eigenparameter boundary
conditions and transmission conditions, J. Diff. Eqns. 255 (2013), no. 9, 2615--2635.
https://doi.org/10.1016/j.jde.2013.07.005


\bibitem{but-half} Buterin S.A. On half inverse problem for differential pencils with the spectral parameter in boundary conditions, Tamkang J. Math. 42(2011), no. 3., 355--364.  https://doi.org/10.5556/j.tkjm.42.2011.912 

\bibitem{khalili} Khalili Ya., Baleanu D. Recovering differential pencils with spectral
boundary conditions and spectral jump
conditions, Journal of Inequalities and Applications 2020 (2020), article 262 (13 pp).
https://doi.org/10.1186/s13660-020-02537-z

\bibitem{durak} Amirov R., Ergun A., Durak S. Half-inverse problems for the quadratic pencil of the
Sturm-Liouville equations with impulse, Num. Meth. PDE 37 (2021), no. 1, 915--924.
https://doi.org/10.1002/num.22559






\bibitem{wang} Wang Yu.P. The inverse spectral problem for differential pencils by mixed spectral data, Applied Mathematics and Computation 338 (2018), 544--551. https://doi.org/10.1016/j.amc.2018.06.052


 \bibitem{yang}  Yang Ch.-F.,  Zettl A. Half inverse problems for quadratic pencils of Sturm-Liouville operators, Taiwanese J. Math. 16 (2012), no. 5, 1829--1846. 
 https://doi.org/10.11650/twjm/1500406800 

\bibitem{bon-graph} Bondarenko N.P. A Partial Inverse Problem for the Differential Pencil on a Star-Shaped Graph, Results Math 72 (2017), 1933--1942. https://doi.org/10.1007/s00025-017-0683-7

\bibitem{wang-bon} Wang Yu.P.,  Bondarenko N.,  Shieh Ch.Ts. The inverse problem for differential pencils on a star-shaped graph with mixed spectral data, Sci China Math 63 (2020), 1559--1570. https://doi.org/10.1007/s11425-018-9485-3

\bibitem{bon-phys} Bondarenko N.P. Inverse problem for the differential pencil on an arbitrary graph with partial information given on the coefficients, 
Anal.Math.Phys. 9 (2019), 1393--1409.
https://doi.org/10.1007/s13324-018-0244-6

 \bibitem{shieh} Bondarenko N.P., Shieh Ch.Ts. Partial inverse problems for quadratic differential pencils on a graph with a loop, Journal of Inverse and Ill-posed Problems 28 (2020), no. 3, 449--463. https://doi.org/10.1515/jiip-2018-0104
 
 
 
 \bibitem{hoh} Hochstadt H. and Lieberman B. An inverse Sturm--Liouville problem with mixed given data, SIAM J. Appl. Math. 34 (1978), issue 4, 676--680.
https://doi.org/10.1137/0134054

\bibitem{OpenMath} Bondarenko N.P. Inverse Sturm-Liouville problem with analytical functions in the boundary condition, Open Mathematics 18 (2020), no. 1, 512--528. https://doi.org/10.1515/math-2020-0188

\bibitem{MethAppl}  Bondarenko N.P. Solvability and stability of the inverse Sturm-Liouville problem with analytical functions in the boundary condition. Math. Meth. Appl. Sci. 43 (2020), issue 11, 7009--7021. https://doi.org/10.1002/mma.6451 

\bibitem{sing-graph} Bondarenko N.P. A partial inverse Sturm-Liouville problem on an arbitrary graph, Mathematical Methods in the Applied Sciences 44 (2021), no.~8, 6896--6910.
https://doi.org/10.1002/mma.7231

\bibitem{horvat} Horvath M. Inverse spectral problems and closed exponential systems, Ann. of Math. 162 (2005), 885--918.
https://doi.org/10.4007/annals.2005.162.885 

\bibitem{torba} Kravchenko V.V., Torba S.M. A practical method for recovering Sturm-Liouville problems from the Weyl function, Inverse Problems 37 (2021), no.~6, 065011.


\bibitem{pronska} Pronska N.I. Asymptotics of eigenvalues and eigenfunctions of energy-dependent Sturm-Liouville equations, Mat. Stud. 40 (2013), no. 1, 38--52. http://matstud.org.ua/texts/2013/40\_1/38-52.html
\bibitem{b-stab} Buterin S.A. On the uniform stability of recovering sine-type functions with asymptotically separated zeros, Mat. Zametki 111 (2022), no. 3, 343--355 (in Russian); English transl.: Math. Notes 111 (2022), no. 3, 15--27. 
https://doi.org/10.1134/S0001434622030026

\bibitem{gohberg} Gohberg I. C. and Krein M.G., Introduction to the Theory of Linear Nonselfadjoint Operators in Hilbert Space, Translations of Mathematical Monographs 18, American Mathematical Society, Providence, RI, 1969.

\bibitem{levin} Levin B.Ja. Distribution of Zeros of Entire Functions. 2nd revised ed. American Mathematical Society, Providence, RI, 1980.

\end{thebibliography}
\end{document}